\documentclass[12pt,a4paper]{article}
\usepackage{amssymb,amsmath,amsthm}
\usepackage{graphicx}

\newcommand\cB{{\mathcal B}}

\newcommand\cF{{\mathcal F}}

\newcommand\cG{{\mathcal G}}

\theoremstyle{plain}
\newtheorem{theorem}{Theorem}[section]
\newtheorem{lemma}[theorem]{Lemma}

\newtheorem{question}[theorem]{Question}
\theoremstyle{definition}

\newtheorem{defn}[theorem]{Definition}
\newtheorem{con}[theorem]{Construction}
\newtheorem{claim}[theorem]{Claim}

\newtheorem*{remark}{Remark}

\newcommand\cref[1]{Corollary~\ref{cor:#1}}

\title{Stability results for vertex Tur\'an problems in Kneser graphs}
\begin{document}

\author{D\'aniel Gerbner$^{a}$, Abhishek Methuku$^{b}$, D\'aniel T. Nagy$^{a}$, \\ Bal\'azs Patk\'os$^{a,}$\thanks{corresponding author.}, M\'at\'e Vizer$^{a}$\\
\small $^a$ Alfr\'ed R\'enyi Institute of Mathematics, Hungarian Academy of Sciences\\
\small P.O.B. 127, Budapest H-1364, Hungary.\\
\small $^b$ Central European University, Department of Mathematics\\
\small Budapest, H-1051, N\'ador utca 9.\\
\medskip
\small \texttt{\{gerbner,nagydani,patkos\}@renyi.hu, \{abhishekmethuku,vizermate\}@gmail.com}
\medskip}
\maketitle

\begin{abstract}
The vertex set of the Kneser graph $K(n,k)$ is $V = \binom{[n]}{k}$ and two vertices are adjacent if the corresponding sets are disjoint. For any graph $F$,  the largest size of a vertex set $U \subseteq V$ such that $K(n,k)[U]$ is $F$-free, was recently determined by Alishahi and Taherkhani, whenever $n$ is large enough compared to $k$ and $F$. In this paper, we determine the second largest size of a vertex set $W \subseteq V$ such that $K(n,k)[W]$ is $F$-free, in the case when $F$ is an even cycle or a complete multi-partite graph. In the latter case, we actually give a more general theorem depending on the chromatic number of $F$. 
\end{abstract}



\section{Introduction}
Tur\'an-type problems are fundamental in extremal (hyper)graph theory. For a pair $H$ and $F$ of graphs, they ask for the maximum number of \textit{edges} that a subgraph $G$ of the host graph $H$ can have without containing the forbidden graph $F$. A variant of this problem is the so-called vertex Tur\'an problem where given a host graph $H$ and a forbidden graph $F$, one is interested in the maximum size of a vertex set $U \subset V(H)$ such that the induced subgraph $H[U]$ is $F$-free.
 
This problem has been studied in the context of several host graphs. In this paper we follow the recent work of Alishahi and Taherkhani \cite{AT}, who determined the exact answer to the vertex Tur\'an problem when $H$ is the \textit{Kneser graph} $K(n,k)$, which is defined on the vertex set $\binom{[n]}{k}=\{K\subseteq [n]=\{1,2,\dots, n\}: |K|=k\}$ where two vertices $K,K'$ are adjacent if and only if $K\cap K'=\emptyset$.

\begin{theorem}[Alishahi, Taherkhani \cite{AT}]\label{at}
For any graph $F$, let $\chi$ denote its chromatic number and let $\eta=\eta(F)$ denote the minimum possible size of a color class of $G$ over all possible proper $\chi$-colorings of $F$. Then for any $k$ there exists an integer $n_0=n_0(k,F)$ such that if $n\ge n_0$ and for a family $\cG\subseteq \binom{[n]}{k}$ the induced subgraph $K(n,k)[\cG]$ is $F$-free, then $|\cG|\le \binom{n}{k}-\binom{n-\chi+1}{k}+\eta -1$. Moreover, if equality holds, then there exists a $(\chi-1)$-set $L$ such that $|\{G\in \cG: G\cap L=\emptyset\}| = \eta-1$.
\end{theorem}

Observe that the vertex Tur\'an problem in the Kneser graph $K(n,k)$ generalizes several intersection problems in $\binom{[n]}{k}$:
\begin{itemize}
\item
If $F=K_2$, the graph consisting a single edge, then the vertex Tur\'an problem asks for the maximum size of an independent set in $K(n,k)$ or equivalently the size of a largest \textit{intersecting} family $\cF\subseteq \binom{[n}{k}$ (i.e. $F\cap F'\neq \emptyset$ for all $F,F'\in \cF$). The celebrated theorem of Erd\H os, Ko, and Rado states that this is $\binom{n-1}{k-1}$ if $2k\le n$ holds. Furthermore, for intersecting families $\cF\subseteq \binom{[n]}{k}$ of size $\binom{n-1}{k-1}$ we have $\cap_{F\in\cF}F\neq \emptyset$ provided $n\ge 2k+1$.
\item
If $F=K_s$ for some $s\ge 3$, then the vertex Tur\'an problem is equivalent to Erd\H os's famous matching conjecture: $K(n,k)[\cF]$ is $K_s$-free if and only if $\cF$ does not contain a \textit{matching} of size $s$ ($s$ pairwise disjoint sets). Erd\H os conjectured that the maximum size of such a family is $\max\{\binom{sk-1}{k}, \binom{n}{k}-\binom{n-s+1}{k}\}$. 
\item
Gerbner, Lemons, Palmer, Patk\'os, and Sz\'ecsi \cite{Getal} considered \textit{$l$-almost intersecting} families $\cF\subseteq \binom{[n]}{k}$ such that for any $F\in \cF$ there are at most $l$  sets in $\cF$ that are disjoint from $F$. This is equivalent to $K(n,k)[\cF]$ being $K_{1,l}$-free. 
\item
Katona and Nagy \cite{KN} considered \textit{$(s,t)$-union intersecting} families $\cF\subseteq \binom{[n]}{k}$ such that for any 
$F_1,F_2,$ $\dots, F_s,$ $F'_1,F'_2,$ $\dots,F'_t\in \cF$ we have $(\cup_{i=1}^sF_i)\cap (\cup_{j=1}^tF'_j)\neq \emptyset$. This is equivalent to $K(n,k)[\cF]$ being $K_{s,t}$-free. 
\end{itemize}

Theorem \ref{at} leads into several directions. One can try to determine the smallest value of the threshold $n_0(k,G)$. Alishahi and Taherkhani \cite{AT} improved the upper bound on $n_0$ for $l$-almost intersecting and $(s,t)$-union intersecting families. Erd\H os's matching conjecture is known to hold if $n\ge (2s+1)k-s$. This is due to Frankl \cite{F2} and he also showed \cite{F} that the conjecture is true if $k=3$.

Another direction is to determine the ``second largest" family with 
$K(n,k)[\cF]$ being $G$-free. In the case of $F=K_2$ this means that we are looking for the largest intersecting family $\cF\subseteq \binom{[n]}{k}$ with $\cap_{F\in \cF}F=\emptyset$. This is the following famous result of Hilton and Milner.

\begin{theorem}[Hilton, Milner \cite{HM}]\label{hm}
If $\cF\subseteq \binom{[n]}{k}$ is an intersecting family with $n\ge 2k+1$ and $\cap_{F\in\cF}F=\emptyset$, then $|\cF|\le \binom{n-1}{k-1}-\binom{n-k-1}{k-1}+1$.
\end{theorem}

In the case of $F=K_{s,t}$ extremal families are not intersecting, so to describe the condition of being ``second largest'' precisely, we introduce the following parameter.

\begin{defn} For a family $\cF$ and integer $t\ge 2$ let $\ell_t(\cF)$ denote the minimum number $m$ such that one can remove $m$ sets from $\cF$ with the resulting family not containing $t$ pairwise disjoint sets. We will write $\ell(\cF)$ instead of $\ell_2(\cF)$. Note that this is the minimum number of sets one needs to remove from $\cF$ in order to obtain an intersecting family.
\end{defn}
Observe that if $s\le t$, then for any family $\cF$ with $\ell(\cF)\le s-1$ the induced subgraph $K(n,k)[\cF]$ is $K_{s,t}$-free. In \cite{AT}, the following asymptotic stability result was proved.

\begin{theorem}[Alishahi, Taherkhani \cite{AT}]\label{at2}
For any integers $s\le t$ and $k$,  and positive real number $\beta$, there exists an $n_0=n_0(k,s,t, \beta)$ such that the following holds for $n\ge n_0$. If for $\cF\subseteq \binom{[n]}{k}$ with $\ell(\cF)\ge s$, the induced subgraph $K(n,k)[\cF]$ is $K_{s,t}$-free, then $|\cF|\le (s+\beta)(\binom{n-1}{k-1}-\binom{n-k-1}{k-1})$ holds.
\end{theorem}

Note that the above bound is asymptotically optimal as shown by any family $\cF_{s,t}=\{F\in \binom{[n]}{k}:1\in F, F\cap S\neq \emptyset\}\cup \{H_1,H_2,\dots,H_s\} \cup \{F'_1,F'_2,\dots,F'_{t-1}\}$, where $S= [2,sk+1]$, $H_i=[(i-1)k+2, ik+1]$ for all $i=1,2,\dots,s$ and $F'_1,F'_2,\dots,F'_{t-1}$ are distinct sets containing 1 and disjoint with $S$.

We improve Theorem \ref{at2} to obtain the following precise stability result for families $\cF$ for which $K(n,k)[\cF]$ is $K_{s,t}$-free.

\begin{theorem}\label{stabst}
For any $2\le s\le t$ and $k$ there exists $n_0=n_0(s,t,k)$ such that the following holds for $n\ge n_0$. If $\cF\subseteq \binom{[n]}{k}$ is a family with $\ell(\cF)\ge s$ and $K(n,k)[\cF]$ is $K_{s,t}$-free, then we have $|\cF|\le \binom{n-1}{k-1}-\binom{n-sk-1}{k-1}+s+t-1$. Moreover, equality holds if and only if $\cF$ is isomorphic to some $\cF_{s,t}$.
\end{theorem}
\bigskip

Using Theorem \ref{stabst}, we obtain a general stability result for the case when $F$ is a complete multi-partite graph. We consider the family $\cF_{s_1,s_2,\dots,s_{r+1}}$ that consists of $s_{r+1}$ pairwise disjoint $k$-subsets $F_1,F_2,\dots, F_{s_{r+1}}$ of $[n]$ that do not meet $[r]$ and those $k$-subsets of $[n]$ that either (i) intersect $[r-1]$ or (ii) contain $r$ and meet $\cup_{j=1}^{s_{r+1}}F_j$ and (iii) $s_r-1$ other $k$-sets containing $r$. Clearly, if $s_1\ge s_2\ge\dots \ge s_r\ge s_{r+1}$ holds, then $K(n,k)[\cF_{s_1,s_2,\dots,s_{r+1}}]$ is $K_{s_1,s_2,\dots,s_{r+1}}$-free and its size is $\binom{n}{k}-\binom{n-r+1}{k}+\binom{n-r}{k-1}-\binom{n- s_{r+1}k-r}{k-1}+s_r+s_{r+1}-1$.

\begin{theorem}\label{stabmulti}
For any $k\ge 2$ and integers $s_1\ge s_2 \ge \dots \ge s_r\ge s_{r+1}\ge 1$ there exists $n_0=n_0(k,s_1,\dots,s_{r+1})$ such that if $n\ge n_0$ and $\cF\subseteq \binom{[n]}{k}$ is a family with $\ell_{r+1}(\cF)\ge s$ and $K(n,k)[\cF]$ is $K_{s_1,s_2,\dots,s_{r+1}}$-free, then we have $|\cF|\le \binom{n}{k}-\binom{n-r+1}{k}+\binom{n-r}{k-1}-\binom{n- s_{r+1}k-r}{k-1}+s_r+s_{r+1}-1$. Moreover, equality holds if and only if $\cF$ is isomorphic to some $\cF_{s_1,s_2,\dots,s_{r+1}}$.
\end{theorem}

Note that Frankl and Kupavskii \cite{FK} proved the special case $s_1=s_2=\dots =s_{r+1}=1$ with the asymptotically best possible threshold $n_0=(2k+o_r(1))(r+1)k$.

\bigskip

Actually, Theorem \ref{stabmulti} is a special case of a more general result that shows that it is enough to solve the stability problem for bipartite graphs. For any graph $F$ with $\chi(F)\ge 3$ let us define $\cB_F$ to be the class of those bipartite graphs $B$ such that there exists a subset $U$ of vertices of $F$ with $F[U]=B$ and $\chi(F[V(F)\setminus U])=\chi(F)-2$. Note that by definition, for any $B\in \cB_F$ we have $\eta(B)\ge \eta(F)$. We define $\cB_{F,\eta}$ to be the subset of those bipartite graphs $B\in \cB_{F}$ for which $\eta(B)=\eta(F)$ holds. To state our result let us introduce some notation. For any graph $F$ let $ex^{(2)}_{v}(n,k,F)$ denote the maximum size of a family $\cF \subseteq \binom{[n]}{k}$ with $\ell_{\chi(F)}(\cF)\ge \eta(F)$ and $K(n,k)[\cF]$ is $F$-free. Observe that Theorem \ref{stabst} is about $ex^{(2)}_{v}(n,k,K_{s,t})$ and Theorem \ref{stabmulti} determines $ex^{(2)}_{v}(n,k,K_{s_1,s_2,\dots,s_{r+1}})$. We define $ex^{(2)}_{v}(n,k,\cB_{F,\eta})$ to be the maximum size of a family $\cF\subseteq\binom{[n]}{k}$ with $\ell_2(\cF)\ge \eta(F)$ such that $K(n,k)[\cF]$ is $B$-free for any $B\in \cB_{F,\eta}$. Similarly, let $\widehat{ex}^{(2)}_{v}(n,k,\cB_{F,\eta})$ be the maximum size of a family $\cF\subseteq\binom{[n]}{k}$ with $\ell_2(\cF)=\eta(F)$ such that $K(n,k)[\cF]$ is $B$-free for any $B\in \cB_{F,\eta}$. Obviously we have $\widehat{ex}^{(2)}_{v}(n,k,\cB_{F,\eta})\le ex^{(2)}_{v}(n,k,\cB_{F,\eta})$ and we do not know any graph $F$ for which the two quantities differ. 

\begin{theorem}\label{chi3}
For any graph with $\chi(F)\ge 3$ there exists an $n_0=n_0(F)$ such that if $n$ is larger than $n_0$, then we have $$\widehat{ex}^{(2)}_{v}(n-\chi(F),k,\cB_{F,\eta})\le ex^{(2)}_{v}(n,k,F)-\left(\binom{n}{k}-\binom{n-\chi(F)+2}{k}\right)\le ex^{(2)}_{v}(n-\chi(F),k,\cB_{F,\eta}).$$
\end{theorem}

Let us remark first that in the case of $F=K_{s_1,s_2,\dots,s_{r+1}}$ we have $\cB_{F}=\{K_{s_i,s_{j}}: 1\le i<j\le r+1\}$ and $\cB_{F,\eta}=\{K_{s_i,s_{r+1}}: 1\le i\le r\}$ and obviously for both families the minimum is taken for $K_{s_r,s_{r+1}}
$, so Theorems \ref{chi3} and \ref{stabst} yield the bound of Theorem \ref{stabmulti}.

\bigskip

In view of Theorem \ref{chi3}, we turn our attention to bipartite graphs, namely to the case of even cycles: $F=C_{2s}$. According to Theorem \ref{at}, the largest families $\cF$ such that $K(n,k)[\cF]$ is $C_{2s}$-free have $\ell(\cF)=s-1$, so once again we will be interested in families for which $\ell(\cF)\ge s$. The case $C_4=K_{2,2}$ is solved by Theorem \ref{stabst} (at least for large enough $n$).  Here we define a construction that happens to be asymptotically extremal for any $s\ge 3$.

\begin{con}\label{cocon}
Let us define $\cG_6\subseteq \binom{[n]}{k}$ as
\[
\cG_6=\left\{G\in \binom{[n]}{k}:1\in G, G\cap [2,2k+1]\neq \emptyset\right\} \cup \{[2,k+1],[k+2,2k+1],[2k+2,3k+1]\}.
\]
So $|\cG_6|=\binom{n-1}{k-1}-\binom{n-2k-1}{k-1}+3$.

For $s\ge 4$ we define the family $\cG_{2s}\subseteq \binom{[n]}{k}$ in the following way: let $K=[2,k+1], K'=[k+2,2k]$ and let $H_1,H_2,\dots, H_{s-1}$ be $k$-sets containing $K'$ and not containing 1. Then
\[
\cG_{2s}=\left\{G\in \binom{[n]}{k}:1\in G, G\cap (K\cup K')\neq \emptyset\right\}\cup\{K, H_1,H_2,\dots,H_{s-1}\}.
\]
So $|\cG_{2s}|=\binom{n-1}{k-1}-\binom{n-2k}{k-1}+s$.
\end{con}

Somewhat surprisingly, it turns out that the asymptotics of the size of the largest family is $(2k+o(1))\binom{n-2}{k-2}$ for $s=2$ and $s=3$ if $k$ is fixed and $n$ tends to infinity, and it is $(2k-1+o(1))\binom{n-2}{k-2}$ for $s\ge 4$.

Observe that $K(n,k)[\cG_{2s}]$ is $C_{2s}$-free and $\ell(\cG_{2s})=s$. Indeed, if $K(n,k)[\cG_{2s}]$ contained a copy of $C_{2s}$, then this copy should contain all $s$ sets not containing 1 as the sets containing 1 form an independent set in $K(n,k)$. In the case $s=3$, $\cF_6$ does not contain any set that is disjoint from both $[2,k+1]$ and $[k+2,2k+1]$, so no $C_6$ exists in $K(n,k)[\cG_6]$. In the case $s\ge 4$, there is no set in $\cG_{2s}$ that is disjoint from both $K$ and $H_i$ for some $i=1,2,\dots, s-1$, so no copy of $C_{2s}$ can exist in $\cG_{2s}$.

The next theorems state that if $n$ is large enough, then Construction \ref{cocon} is asymptotically optimal. Moreover, as the above proofs show that $K(n,k)[\cG_{2s}]$ does not even contain a \textit{path} on $2s$ vertices, Construction \ref{cocon} is asymptotically optimal for the problem of forbidding paths as well.

\begin{theorem}\label{cycle6}
For any $k \ge 2$, there exists $n_0=n_0(k)$ with the following property: if $n\ge n_0$ and $\cF\subseteq \binom{[n]}{k}$ is a family with $\ell(\cF)\ge 3$ and $K(n,k)[\cF]$ is $C_{6}$-free, then we have $|\cF|< \binom{n-1}{k-1}-\binom{n-2k-1}{k-1}+10^6(\binom{n-1}{k-1}-\binom{n-2k-1}{k-1})^{3/4}$.
\end{theorem}

\begin{theorem}\label{cycles}
For any $s\ge 4$ and $k \ge 3$ there exists $n_0=n_0(k,s)$ such if $n\ge n_0$ and $\cF\subseteq \binom{[n]}{k}$ is a family with $\ell(\cF)\ge s$ and $K(n,k)[\cF]$ is $C_{2s}$-free, then we have $|\cF|\le \binom{n-1}{k-1}-\binom{n-2k}{k-1}+(k^2+1)\binom{n-3}{k-3}$.
\end{theorem}

Let us finish the introduction by a remark on the second order term in Theorem \ref{cycles}. 
\begin{remark}
If $s-1\le k$, then the family $\cG_{2s}$ can be extended to a family $\cG_{2s}^+\cup \cG_{2s}$ so that $K(n,k)[\cG_{2s}^+\cup \cG_{2s}]$ is still $C_{2s}$-free. Suppose the sets $H_1,H_2,\dots, H_{s-1}$ are all disjoint from $K$, say $H_i=K'\cup \{2k+i\}$ for $i=1,2\dots,s-1$. Then we can define 
$$\cG_{2s}^+=\left\{G\in \binom{[n]}{k}: \{1,2k+1,2k+2,\dots,2k+s-2\}\subseteq G\right\}$$ and observe that $K(n,k)[\cG_{2s}\cup \cG^+_{2s}]$ is still $C_{2s}$-free. Indeed, a copy of $C_{2s}$ would have to contain $K,H_1,H_2,\dots, H_{s-1}$ as other vertices form an independent set.  Moreover, $K$ and $H_i$ have a common neighbour in $\cG_{2s}\cup \cG_{2s}^+$ if and only if $i=s-1$, so $K$ cannot be contained in $C_{2s}$.

Clearly, $|\cG_{2s}^+\setminus \cG_{2s}|=\binom{n-k-s+1}{k-s+1}$, so in particular if $s=4$, then the order of magnitude of the second order term in Theorem \ref{cycles} is sharp (when $n$ is large enough compared to $k$).
\end{remark}

All our results resemble the original Hilton-Milner theorem in the following sense. In Theorem \ref{stabst}, Theorem \ref{cycle6}, Theorem \ref{cycles}, almost all sets of the (asymptotically) extremal family share a common element $x$ and meet some set $S$ ($x\notin S$) of fixed size. We wonder whether this phenomenon is true for all bipartite graphs.

\begin{question}
Is it true that for any bipartite graph $B$ and integer $k\ge 3$ there exists an integer $s$ such that the following holds:
\begin{itemize}
\item
for any family $\cF\subseteq \binom{[n]}{k}$ with $\ell(\cF)\ge \eta(B)$ if $K(n,k)[\cF]$ is $B$-free, then $|\cF|\le \binom{n-1}{k-1}-\binom{n-1-s}{k-1}+o(n^{k-2})$
\item
the family $\{G\in \binom{[n]}{k}:1\in G, G\cap [2,s+1]\neq \emptyset\}$ is contained in a family $\cG\subseteq \binom{[n]}{k}$ with $\ell(\cG)\ge \eta(B)$ such that $K(n,k)[\cG]$ is $B$-free.
\end{itemize}
\end{question}

\section{Proofs}
Let us start this section by stating the original Tur\'an number results on the maximum number of edges in $K_{s,t}$-free and $C_{2s}$-free graphs.
\begin{theorem}[K\H ov\'ari, S\'os, Tur\'an \cite{KST}]\label{kst}
For any pair $1\le s\le t$ of integers if a graph $G$ on $n$ vertices is $K_{s,t}$-free, then $e(G)\le (1/2+o(1))(t-1)^{1/s}n^{2-\frac{1}{s}}$ holds.
\end{theorem}

\begin{theorem}[Bondy, Simonovits \cite{BS}]\label{bs}
If $G$ is a graph on $n$ vertices that does not contain a cycle of length $2s$, then $e(G)\le 100sn^{1+1/s}$ holds.
\end{theorem}

We will also need the following lemma by Balogh, Bollob\'as and Narayanan. (It was improved by a factor of 2 in \cite{AT}, but for our purposes the original lemma will be sufficient.)

\begin{lemma}[Balogh, Bollob\'as, Narayanan \cite{BBN}]\label{bbn}
For any family $\cF\subseteq \binom{[n]}{k}$ we have $e(K(n,k)[\cF])\ge \frac{l(\cF)^2}{2\binom{2k}{k}}$.
\end{lemma}

We start with the following simple lemma.

\begin{lemma}\label{easy}
Let $s\le t$ and let $H_1,H_2,\dots, H_s, H_{s+1}$ be sets in $\binom{[n]}{k}$ and $x\in [n]\setminus \cup_{i=1}^{s+1}H_i$. Suppose that $\cF \subseteq \{F\in \binom{[n]}{k}: x\in F\}$ such that for $\cF':=\cF\cup \{H_1,H_2,\dots, H_{s+1}\}$ the induced subgraph $K(n,k)[\cF']$ is $K_{s,t}$-free. Then there exists $n_0=n_0(k,s,t)$ such that if $n\ge n_0$ holds, then we have $$|\cF|\le \binom{n-1}{k-1}-\binom{n-\lfloor\frac{(s+1)
k}{2}\rfloor-1}{k-1}+(s+1)(t-1).$$
\end{lemma}

\begin{proof}
The number of sets in $\cF$ that meet at most one $H_j$ is at most $(s+1)(t-1)$ as $K(n,k)[\cF']$ is $K_{s,t}$-free.
Let us define $T=\{y\in [n]: \exists i\neq j \hskip 0.3truecm y\in H_i\cap H_j\}$. Those sets in $\cF$ that meet at least two of the $H_j$'s must either a) intersect $T$ or b) intersect at least two of the $(H_j\setminus T)$'s. Clearly, $|T|\le \lfloor\frac{(s+1)k}{2}\rfloor$, so the number of sets in $\cF$ meeting $T$ is at most $\binom{n-1}{k-1}-\binom{n-1-|T|}{k-1}\le \binom{n-1}{k-1}-\binom{n-\lfloor\frac{(s+1)
k}{2}\rfloor-1}{k-1}=:B$.

Assume first $|T|<\lfloor\frac{(s+1)
k}{2}\rfloor$, then $B-(\binom{n-1}{k-1}-\binom{n-1-|T|}{k-1})=\Omega(n^{k-2})$. Observe that the number of sets in $\cF$ that are disjoint with $T$ and meet at least two $H_j\setminus T$ is at most $\sum_{i,j}|H_i\setminus T|\cdot|H_j\setminus T|\binom{n-3}{k-3}\le \binom{s+1}{2}k^2\binom{n-3}{k-3}=O(n^{k-3})$. Therefore if $n$ is large enough, then $|\cF|\le \binom{n-1}{k-1}-\binom{n-\lfloor\frac{(s+1)
k}{2}\rfloor-1}{k-1}-\varepsilon n^{k-2}$ for some $\varepsilon>0$.

Assume now $T=\lfloor\frac{(s+1)
k}{2}\rfloor$. This implies that at most one of the $H_j\setminus T$ is non-empty, so $\cF$ does not contain sets of type b). Thus we have $|\cF|\le B+(s+1)(t-1)$. 
\end{proof}



Now we are ready to prove our main result on families $\cF\subseteq \binom{[n]}{k}$ with $K(n,k)[\cF]$ being $K_{s,t}$-free.

\begin{proof}[Proof of Theorem \ref{stabst}]
Let $\cF\subseteq \binom{[n]}{k}$ be a family such that $K(n,k)[\cF]$ is $K_{s,t}$-free and $|\cF|= \binom{n-1}{k-1}-\binom{n-sk-1}{k-1}+s+t-1$. We consider three cases according to the value of $\ell(\cF)$.

\bigskip

\textsc{Case I}: $\ell(\cF)=s$.

\smallskip

Consider $F_1,F_2,\dots, F_s\in \cF$ such that $\cF'=\cF\setminus \{F_i: 1\le i\le s\}$ is intersecting. Then, as $$|\cF'|=\binom{n-1}{k-1}-\binom{n-sk-1}{k-1}+t-1> \binom{n-1}{k-1}-\binom{n-k-1}{k-1},$$ Theorem \ref{hm} implies that the sets in $\cF'$ share a common element. Since $K(n,k)[\cF]$ is $K_{s,t}$-free $\cF'$ can contain at most $t-1$ sets disjoint from $T:=\cup_{i=1}^sF_i$. So the size of $\cF$ is at most $$\binom{n-1}{k-1}-\binom{n-|T|-1}{k-1}+t-1+s\le \binom{n-1}{k-1}-\binom{n-sk-1}{k-1}+s+t-1$$ with equality if and only if $\cF$ is isomorphic to some $\cF_{s,t}$.

\bigskip

\textsc{Case II}: $s+1\le \ell(\cF)\le (\binom{n-1}{k-1}-\binom{n-sk-1}{k-1})^{1-\frac{1}{3s}}$.

\smallskip

Let $\cF'$ be a largest intersecting subfamily of $\cF$. As the size of $\cF'$ is $\binom{n-1}{k-1}-\binom{n-sk-1}{k-1}+s+t-1-l(\cF)$ which is larger than $\binom{n-1}{k-1}-\binom{n-k-1}{k-1}+1$ if $n$ is large enough, Theorem \ref{hm} implies that the sets in $\cF'$ share a common element. Let us apply Lemma \ref{easy} to $\cF'$ and $s+1$ sets $F_1,F_2,\dots, F_{s+1} \in \cF\setminus \cF'$ to obtain $$|\cF'|\le \binom{n-1}{k-1}-\binom{n-\frac{(s+1)
k}{2}-1}{k-1}+(s+1)(t-1).$$ Therefore, we have 
\[
|\cF|\le \binom{n-1}{k-1}-\binom{n-\frac{(s+1)
k}{2}-1}{k-1}+(s+1)(t-1) +\left(\binom{n-1}{k-1}-\binom{n-sk-1}{k-1}\right)^{1-\frac{1}{3s}},
\]
which is smaller than $\binom{n-1}{k-1}-\binom{n-sk-1}{k-1},$ if $n$ is large enough.
\bigskip

\textsc{Case III}: $\left(\binom{n-1}{k-1}-\binom{n-sk-1}{k-1}\right)^{1-\frac{1}{3s}}\le \ell(\cF)$.

\smallskip

Then by Lemma \ref{bbn}, we have $$e(K(n,k)[\cF])\ge \frac{\left(\binom{n-1}{k-1}-\binom{n-sk-1}{k-1}\right)^{2-\frac{2}{3s}}}{2\binom{2k}{k}}.$$ For large enough $n$, this is larger than $(1/2+o(1))(t-1)^{\frac{1}{s}}|\cF|^{2-\frac{1}{s}}$, so $K(n,k)[\cF]$ contains $K_{s,t}$ by Theorem \ref{kst}.
\end{proof}

\bigskip

\begin{proof}[Proof of Theorem \ref{stabmulti}]
Let $\cF \subseteq \binom{[n]}{k}$ be a family of size $ \binom{n}{k}-\binom{n-r+1}{k}+\binom{n-r}{k-1}-\binom{n- s_{r+1}k-r}{k-1}+s_r+s_{r+1}-1$ with $\ell_{r+1}(\cF)\ge s_{r+1}$ such that $K(n,k)[\cF]$ is $K_{s_1,s_2,\dots,s_{r+1}}$-free. The proof proceeds by a case analysis according to the number of large degree vertices. We say that $x\in [n]$ has \textit{large degree} if $\cF_x=\{F\in \cF: x\in F\}$ has size at least $d=\binom{n-1}{k-1}-\binom{n-Q  k-1}{k-1}+Q$ where $Q:=\sum_{i=1}^{r+1}s_i$. Let $D$ denote the set of large degree vertices. We will use the following claim in which $G_1 \oplus G_2$ denotes the \textit{join} of $G_1$ and $G_2$, i.e. the graph consisting of disjoint copies of $G_1$ and $G_2$ with all possible edges between the $G_1$ and $G_2$.

\begin{claim}\label{extend}
Suppose $\cF$ contains a subfamily $\cG\subseteq \binom{[n]\setminus D}{k}$ with $|\cG|\le Q-\sum_{i=1}^{|D|}s_i$ and $K(n,k)[\cG]$ is isomorphic to $G$, then $K(n,k)[\cF]$ contains $K_{s_1,s_2,\dots,s_{|D|}} \oplus G$.
\end{claim}

\begin{proof}[Proof of Claim]
Note that $d$ is $Q$ plus the number of $k$-subsets of $[n]$ containing a fixed element $x$ of $[n]$ and meeting a set $S$ of size $Qk$. As $K_{s_1,s_2,\dots,s_{|D|}} \oplus G$ contains at most $Qk$ vertices, we can pick the sets corresponding to $K_{s_1,s_2,\dots,s_{|D|}}$ greedily. Indeed, for each high degree vertex, we can choose $s_i$ sets containing it which avoid the set spanned by the already chosen sets and the (at most $Q$) sets corresponding to $G$.
\end{proof}

\smallskip

\textsc{Case I:} $|D|\ge r$.

\smallskip

Let $D'\subset D$ be of size $r$ and let $F_1,F_2,\dots, F_{s_{r+1}}$ be sets in $\cF$ not meeting $D'$. (There exists such sets as otherwise $\ell_{r+1}(\cF)<s_{r+1}$ would hold.) Applying Claim \ref{extend} with $\cG=\{F_1,F_2,\dots,F_{s_{r+1}}\}$ we obtain that $K(n,k)[\cF]$ is not $K_{s_1,s_2,\dots, s_{r+1}}$-free.

\

\textsc{Case II:} $|D|= r-1$.

\smallskip

Then $\cF'=\cF\setminus \cup_{x\in D}\cF_x\subseteq \binom{[n]\setminus D}{k}$ has size at least $\binom{n-r}{k-1}-\binom{n-r-s_{r+1}k}{k-1}+s_r+s_{r+1}-1$ with equality if and only if $\cup_{x\in D}\cF_x$ contains all $k$-sets meeting $D$. Either $K(n,k)[\cF']$ contains $K_{s_r,s_{r+1}}$ and thus, by Claim \ref{extend}, $\cF$ contains $K_{s_1,s_2,\dots,s_{r+1}}$. Otherwise note that $\ell_{r+1}(\cF)\ge s_{r+1}$ implies $\ell_2(\cF')=\ell(\cF')\ge s_{r+1}$, so Theorem \ref{stabst} implies that if $n$ is large enough, then $\cF'$ is some $\cF_{s_r,s_{r+1}}$ and thus, $\cF$ is some $\cF_{s_1,s_2,\dots,s_{r+1}}$.

\ 

\textsc{Case III:} $|D|\le r-2$.

\smallskip

In this case $\cF'=\cF\setminus \cup_{x\in D}\cF_x\subseteq \binom{[n]\setminus D}{k}$ has size at least $\binom{n-r+1}{k-1}+\binom{n-r}{k-1}-\binom{n-r-s_{r+1}k}{k-1}+s_r+s_{r+1}-1$. The order of magnitude of this is $n^{k-1}$, thus it is larger than $Qkd$ if $n$ is large enough. We claim that $K(n,k)[\cF']$ contains $K_Q$ and therefore a copy of $K_{s_1,s_2,\dots,s_{r+1}}$. Indeed, for any $F\in \cF'$ there are at most $kd$ sets in $\cF'$ that intersect $F$, thus we can pick $Q$ pairwise disjoint sets greedily.

\end{proof}

\begin{proof}[Proof of Theorem \ref{chi3}]

First we show the construction for the lower bound. For a graph $F$ with $\chi(F)\ge 3$, let $\cG_F\subseteq \binom{[n-\chi(F)+2]}{k}$ be a family of size $\widehat{ex}^{(2)}_{v}(n-\chi(F)+2,k,\cB_{F,\eta})$ such that $K(n-\chi(F)+2,k)[\cG_F]$ is $B$-free for any $B\in \cB_{F,\eta}$ and $\ell(\cG_F)=\eta(F)$. Let us define $\cF_F\subseteq \binom{[n]}{k}$ as $$\cF_F=\cG_F\cup \left\{K\in \binom{[n]}{k}:K \cap [n-\chi(F)+3,n]\neq \emptyset\right\}.$$
Clearly, we have 
$$|\cF_F|=\binom{n}{k}-\binom{n-\chi(F)+2}{k}+\widehat{ex}^{(2)}_{v}(n-\chi(F)+2,k,\cB_{F,\eta})$$ 
and we claim that $K(n,k)[\cF_F]$ is $F$-free. Indeed, if $K(n,k)[\cF_F]$ contains $F$, then $K(n,k)[\cG_F]$ contains some $B\in \cB_F$, as $ \{K\in \binom{[n]}{k}:K \cap [n-\chi(F)+3,n]\neq \emptyset\}$ is the union $\chi(F)-2$ intersecting families. This is impossible for $B\in \cB_{F,\eta}$ by definition of $\cG_F$, and it is also impossible for $B\in \cB_F\setminus \cB_{F,\eta}$ as $\ell(\cG_F)=\eta(F)<\eta(B)$.

The proof of the upper bound is basically identical to that of the upper bound in Theorem \ref{stabmulti}, so we just outline it.
Let $\cF\subseteq \binom{[n]}{k}$ with $\ell_{\chi(F)}(\cF)\ge \eta(F)$ and $|\cF|\ge \binom{n}{k}-\binom{n-\chi(F)+2}{k}$ be such that $K(n,k)[\cF]$ is $F$-free. Let us define $d=\binom{n-1}{k}-\binom{n-|v(F)|k-1}{k}+|V(F)|$ and let $D\subseteq V(F)$ be the set of vertices with degree at least $d$ in $\cF$.

\

\textsc{Case I:} $|D|\ge \chi(F)-1$.

\smallskip 

Then one can pick sets of $\cF$ greedily to form a copy of $F$ in $K(n,k)[\cF]$, a contradiction.

\ 

\textsc{Case II:} $|D|=\chi(F)-2$.

\smallskip 

Then $\cF'=\{K\in \cF:K\cap D\neq \emptyset\}$ has size at most $\binom{n}{k}-\binom{n-\chi(F)+2}{k}$. Also $K(n,k)[\cF\setminus \cF']$ cannot contain any $B\in \cB_{F,\eta}$, as otherwise $K(n,k)[\cF]$ would contain $F$. Observe that $\ell_{\chi(F)}(\cF)\ge \eta(F)$ implies $\ell(\cF\setminus \cF')\ge \eta(F)$, so we have $|\cF\setminus \cF'|\le ex^{(2)}_{v}(n,k,\cB_{F,\eta})$.

\ 

\textsc{Case III:} $|D|\le \chi(F)-3$.

\smallskip 

Then $\cF'=\{K\in \cF:K\cap D\neq \emptyset\}$ has size at most $\binom{n}{k}-\binom{n-\chi(F)+3}{k}$. Therefore $\cF\setminus \cF'$ is of size at least $\binom{n-\chi(F)+2}{k-1}$. If $n$ is large enough compared to $k$, then one can pick greedily a copy of $K_{|V(F)|}$ in $K(n,k)[\cF\setminus \cF']$.
\end{proof}

\bigskip

Now we turn our attention to proving theorems on families that induce cycle-free subgraphs in the Kneser graph.

\begin{proof}[Proof of Theorem \ref{cycle6}]
Let $\cF\subseteq \binom{[n]}{k}$ be a family of subsets such that $K(n,k)[\cF]$ is $C_6$-free, $\ell(\cF)\ge 3$ and $|\cF|= \binom{n-1}{k-1}-\binom{n-2k-1}{k-1}+10^6 (\binom{n-1}{k-1}-\binom{n-2k-1}{k-1})^{3/4}$. 

\bigskip

\textsc{Case I:} $\ell(\cF)\le \frac{10^6}{2}(\binom{n-1}{k-1}-\binom{n-2k-1}{k-1})^{3/4}$.

\smallskip

Let $H_1,H_2,\dots,H_{\ell(\cF)}$ be sets in $\cF$ such that $\cF':=\cF\setminus \{H_1,H_2,\dots,H_{\ell(\cF)}\}$ is intersecting. Then as $$|\cF'|\ge \binom{n-1}{k-1}-\binom{n-2k-1}{k-1}+\frac{10^6}{2}\left(\binom{n-1}{k-1}-\binom{n-2k-1}{k-1}\right)^{3/4}>\binom{n-1}{k-1}-\binom{n-k-1}{k-1}+1,$$ Theorem \ref{hm} implies that the sets in $\cF'$ share a common element $x$. The $H_i$'s do not contain this $x$, since $\cF'$ is a largest intersecting family in $\cF$. As $|H_i\cup H_j|\le 2k$ and $$|\cF'|\ge \binom{n-1}{k-1}-\binom{n-2k-1}{k-1}+\frac{10^6}{2}\left(\binom{n-1}{k-1}-\binom{n-2k-1}{k-1}\right)^{3/4},$$ for any $i\neq j$ there exist 3 sets $F_{i,j,1},F_{i,j,2},F_{i,j,3} \in \cF'$ that are disjoint from $H_i\cup H_j$. So we can find a copy of $C_6$ in $\cF$, in which the sets $H_1,H_2,H_3$ represent three independent vertices and the other three sets can be chosen from $\{F_{i,j,k}: 1\le i<j\le 3, 1\le k\le 3\}$ greedily.

\bigskip

\textsc{Case II:} $\ell(\cF)\ge \frac{10^6}{2}(\binom{n-1}{k-1}-\binom{n-2k-1}{k-1})^{3/4}$.

\smallskip

By Lemma \ref{bbn} $K(n,k)[\cF]$ contains at least $\frac{10^{12}}{8\binom{2k}{k}}(\binom{n-1}{k-1}-\binom{n-2k-1}{k-1})^{3/2}$ edges and when $n$ is large enough, this is bigger than $300|\cF|^{4/3}$, so by Theorem \ref{bs} it contains a copy of $C_6$, as desired.
\end{proof}

\begin{proof}[Proof of Theorem \ref{cycles}] Let $\cF\subseteq \binom{[n]}{k}$ be a family of subsets such that $K(n,k)[\cF]$ is $C_{2s}$-free, $\ell(\cF)\ge 3$ and $|\cF|= \binom{n-1}{k-1}-\binom{n-2k}{k-1}+(k^2+1)\binom{n-3}{k-3}$.

\bigskip 

\textsc{Case I}: $\ell(\cF)\le 20s2^k(\binom{n-1}{k-1}-\binom{n-2k}{k-1})^{\frac{s+1}{2s}}$.

\smallskip 

Let $H_1,H_2,\dots,H_{\ell(\cF)}$ be sets in $\cF$ such that $\cF':=\cF\setminus \{H_1,H_2,\dots,H_{\ell(\cF)}\}$ is intersecting. Then as $|\cF'|\ge \binom{n-1}{k-1}-\binom{n-2k}{k-1}+(k^2+1)\binom{n-3}{k-3}-20s2^k(\binom{n-1}{k-1}-\binom{n-2k}{k-1})^{\frac{s+1}{2s}}>\binom{n-1}{k-1}-\binom{n-k-1}{k-1}+1$, Theorem \ref{hm} implies that the sets in $\cF'$ share a common element $x$. The $H_i$'s do not contain this $x$, since $\cF'$ is a maximal intersecting family in $\cF$. Let us define the following auxiliary graph $\Gamma$ with vertex set $\{H_1,H_2,\dots,H_s\}$: two sets $H_i,H_j$ are adjacent if and only if there exist $s$ sets in $\cF'$ that are disjoint from $H_i\cup H_j$. Observe that if $\Gamma$ contains a Hamiltonian cycle, then $\cF$ contains a copy of $C_{2s}$. Indeed, if $H_{\sigma(1)},H_{\sigma(2)},\dots, H_{\sigma(s)}$ is a Hamiltonian cycle, then for any pair  $H_{\sigma(i)},H_{\sigma(i+1)}$ (with $s+1=1$) we can greedily pick different sets $F_i\in\cF'$ with $F_i\cap (H_{\sigma(i)}\cup H_{\sigma(i+1)})=\emptyset$ to get $H_{\sigma(1)}, F_1, H_{\sigma(2)},F_2,\dots, H_{\sigma_(s)},F_s$ a copy of $C_{2s}$ in $K(n,k)[\cF]$. Therefore the next claim and Dirac's theorem \cite{D} finishes the proof of Case I.

\begin{claim}
The minimum degree of $\Gamma$ is at least $s-2$.
\end{claim}

\begin{proof}[Proof of Claim]
First note that if $H_i$ and $H_j$ are not joined in $\Gamma$, then they must be disjoint. Indeed, otherwise $|H_i\cup H_j|\le 2k-1$ and as $|\cF'|\ge \binom{n-1}{k-1}-\binom{n-2k}{k-1}+s$, there are at least $s$ sets in $\cF'$ avoiding $H_i\cup H_j$. Now assume for contradiction that $H_1$ is not connected to $H_2$ and $H_3$, so in particular $H_1\cap (H_2\cup H_3)=\emptyset$. Observe the following
\begin{itemize}
\item
there are at most $s-1$ sets in $\cF'$ that avoid $H_1\cup H_2$ and another $s-1$ sets avoiding $H_1\cup H_3$,
\item
as $|H_1 \cup (H_2\cap H_3)|\le 2k-1$, there are at most $\binom{n-1}{k-1}-\binom{n-2k}{k-1}$ sets in $qcF'$ that meet $H_1 \cup (H_2\cap H_3)$.
\end{itemize}

So there are at least $(k^2+1)\binom{n-3}{k-3}-20s2^k(\binom{n-1}{k-1}-\binom{n-2k}{k-1})^{\frac{s+1}{2s}}$ sets of $\cF'$ containing at least one element $h_2\in H_2\setminus H_3$ and one element $h_3\in H_3\setminus H_2$. Since the number of such pairs is at most $k^2$, there exists a pair $h_2,h_3$ such that the number of sets in $\cF'$ containing both $h_2,h_3$ is more than $\binom{n-3}{k-3}$. But this is clearly impossible as the total number of $k$-sets containing $x,h_2,h_3$ is $\binom{n-3}{k-3}$.
\end{proof}
\textsc{Case II:} $\ell(\cF)\ge 20s2^k(\binom{n-1}{k-1}-\binom{n-2k}{k-1})^{\frac{s+1}{2s}}$.

\smallskip

By Lemma \ref{bbn} $K(n,k)[\cF]$ contains at least $\frac{400s^2 2^{2k}}{2\binom{2k}{k}}(\binom{n-1}{k-1}-\binom{n-2k}{k-1})^{\frac{s+1}{s}} > 100s |\cF|^{1+1/s}$ edges, and thus by Theorem \ref{bs} it contains a copy of $C_{2s}$.
\end{proof}

\bigskip

\textbf{Funding}: Research supported by the \'{U}NKP-17-3 New National Excellence Program of the Ministry of Human Capacities, by National Research, Development and Innovation Office - NKFIH under the grants SNN 116095 and K 116769, by the J\'anos Bolyai Research Fellowship of the Hungarian Academy of Sciences and the Taiwanese-Hungarian Mobility Program of the Hungarian Academy of Sciences.

\end{document}